\UseRawInputEncoding
\documentclass[12pt,reqno]{amsart}
\usepackage{amssymb,amsmath,graphicx,
	amsfonts,euscript}
\usepackage{color}
\usepackage{mathrsfs}
\theoremstyle{plain}

\newtheorem{theorem}{Theorem}[section]

\newtheorem{remark}[theorem]{Remark}
\numberwithin{equation}{section}

\begin{document}

\title[The singularities for a periodic transport equation]{The singularities for a periodic transport equation}

\author[Yong Zhang, Fei Xu, Fengquan Li]{Yong Zhang, Fei Xu, Fengquan Li}

\address[Yong Zhang]{ School of Mathematical Sciences, Dalian University
of Technology, Dalian, 116024, China} \email{18842629891@163.com}

\address[Fei Xu]{School of Mathematical Sciences, Dalian University
of Technology, Dalian, 116024, China}

\address[Fengquan Li]{School of Mathematical Sciences, Dalian University
of Technology, Dalian, 116024, China}

\begin{abstract}
 In this paper, we consider a 1D periodic transport equation with nonlocal flux and fractional dissipation
$$
u_{t}-(Hu)_{x}u_{x}+\kappa\Lambda^{\alpha}u=0,\quad (t,x)\in R^{+}\times S,
$$
where $\kappa\geq0$, $0<\alpha\leq1$ and $S=[-\pi,\pi]$. We first establish the local-in-time well-posedness for this transport equation in $H^{3}(S)$. In the case of $\kappa=0$, we deduce that the solution, starting from the smooth and odd initial data, will develop into singularity in finite time. If adding a weak dissipation term $\kappa\Lambda^{\alpha}u$, we also prove that the finite time blowup would occur.
	\end{abstract}
\maketitle

\section{Introduction}

We will discuss the following equations
 \begin{equation}\label{e11}
 \left\{ \begin{array}{ll}
 u_{t}-(H^{a}\partial_{x}^{b}u)u_{x}+\kappa\Lambda^{\alpha}u=0, ~~~ (t,x)\in R^{+}\times S, \\
u(0,x)=\varphi(x),
 \end{array} \right.
 \end{equation}
 where $a,b\in\{0,1\}$, $Hu(x)$ is the Hilbert transform of $u(x)$ defined by (see \cite{13})
 $$
 Hu(x)=\frac{1}{2\pi}P.V.\int^{\pi}_{-\pi}\frac{u(x-y)}{tan\frac{y}{2}}dy
 $$
 and $H^{0}$ denotes the identity operator.
$\Lambda^{\alpha}u=(-\partial_{xx})^{\frac{\alpha}{2}}u$ is defined by the Fourier transform
$$
\widehat{\Lambda^{\alpha}u}(\xi)=|\xi|^{\alpha}\hat{u}(\xi).
$$
Here $\kappa\geq0$ is viscosity coefficient which controls the strength of dissipation and $0<\alpha\leq2$ controls the magnitude of the dissipation.

If $a=b=0$, (\ref{e11}) is transformed into the classical Brugers equation with fractional dissipation. This equation has been investigated in \cite{1,2,3}. The authors in \cite{1,2} proved the occurrence of shock-type singularities in supercritical case with responding to $0<\alpha<1$ and the global existence in critical case corresponding with $\alpha=1$ and subcritical case corresponding with $1<\alpha\leq2$. In \cite{3}, the authors studied comprehensively the existence, uniqueness, blow up and regularity properties of solutions.

If $a=1, b=0$, (\ref{e11}) can be written a famous nonlocal transport equation, which was firstly considered by A.C\'{o}rdoba et. in \cite{4}, where the authors concluded that the solutions would develop into gradient blow-up as $\kappa=0$ and exist globally as $\kappa>0, \alpha>1$. References \cite{5,6} are recommended for some recent results on this transport equation.

 If $a=0, b=1$, (\ref{e11}) can be reduced to
\begin{equation}\label{e12}
 \left\{ \begin{array}{ll}
 u_{t}-u^{2}_{x}+\kappa\Lambda^{\alpha}u=0, ~~~~ (t,x)\in R^{+}\times S, \\
 u(0,x)=\varphi(x),
 \end{array} \right.
 \end{equation}
 which can be regarded as a fractional order heat equation with a special nonlinearity term $u^{2}_{x}$. On the other hand, it also can be viewed as another form of Burgers equation. For example, taking derivative to (\ref{e12}) with respect to $x$ and letting $v=u_{x}$, then we obtain that $v_{t}-2vv_{x}+\kappa\Lambda^{\alpha}v=0$. Moreover, we can deduce
 $$
 \frac{d}{dt}\|u(t)\|_{H^{1}[\pi,\pi]}\leq0,
 $$
which implies that $\|u(t)\|_{L^{\infty}[-\pi,\pi]}\leq\|u(t)\|_{H^{1}[-\pi,\pi]}\leq\|\varphi\|_{H^{1}[-\pi,\pi]}$. Hence the solution to (\ref{e12}) would not blow up.

In this paper, we mainly consider the case of $a=b=1$, that is to say
 \begin{equation}\label{e13}
 \left\{ \begin{array}{ll}
 u_{t}-(Hu)_{x}u_{x}+\kappa\Lambda^{\alpha}u=0, ~~~~ (t,x)\in R^{+}\times S, \\
 u(0,x)=\varphi(x),
 \end{array} \right.
 \end{equation}
 which is a simplified equation of a class of ill-posed problems arising in the theory of vortex sheets (see \cite{12,14}). As before, if taking derivative to (\ref{e13}) with respect to $x$ and letting $\theta=u_{x}$, we would get
$$
 \left\{ \begin{array}{ll}
 \theta_{t}-(\theta H\theta)_{x}+\kappa\Lambda^{\alpha}\theta=0, ~~~~ (t,x)\in R^{+}\times S, \\
 \theta(0,x)=f(x),
 \end{array} \right.
$$
where $f(x)=\varphi_{x}(x)$. This is a 1D model of the quasi-geostrophic equation (see \cite{7,8,9,10,11}).
 In \cite{9}, the authors showed that there is no $C^{1}([-\pi,\pi]\times[0,+\infty))$ solutions in the case of $\kappa=0$ or $\kappa>0, \alpha=1$. In fact, our results in this paper are consistent with the conclusion in \cite{9}. Besides, we also make up the gap of $\kappa>0, 0<\alpha<1$ by using a different method. For the non-periodic case, the global existence for strictly positive initial data $f(x)>0$ has been proved in \cite{10}. Other important results on this equation can be seen in \cite{11}.

The paper is organized as follows. In section 2, we first give the local well-posedness result of equation (\ref{e13}), and then state the main theorems on blow-up. In section 3, we try to reduce the equation (\ref{e13}) to a family of ordinary differential equations under a class of given initial data. Section 4 is devoted to the proof of the main results on blow-up.

 \section{Local well-posedness and main results}\label{Sec.2}

Before stating to state the main results, we first discuss the local well-posedness result of the problem (\ref{e13}). Combining a priori bound with the compactness argument, one may be able to establish local well-posedness for (\ref{e13}) in $H^{3}[-\pi,\pi]$. The details of proof are the same as that of \cite{15}. Below we merely show how to obtain a priori bound for $\|u\|_{H^{3}[-\pi,\pi]}$, and other parts of the proof are omitted.\\

\begin{theorem}\label{th21}
If $\varphi\in H^{3}[-\pi,\pi]$, then there exists a $T>0$ such that the problem (\ref{e13}) has a unique solution $u\in C([0,T),H^{3}[-\pi,\pi])\cap C^{1}([0,T),H^{2}[-\pi,\pi])$.
\end{theorem}

\begin{proof}
Taking three derivatives to equation (\ref{e13}) and multiplying by $u_{xxx}$ over $[\pi,\pi]$, we have that
\begin{align}\label{e21}
\frac{1}{2}\frac{d}{dt}\|u_{xxx}\|_{L^{2}[\pi,\pi]}^{2}
&=\int^{\pi}_{-\pi}[(Hu)_{x}u_{x}]_{xxx}u_{xxx}dx-\kappa\|\Lambda^{\frac{\alpha}{2}}u_{xxx}\|_{L^{2}[\pi,\pi]}^{2}  \nonumber\\
&\leq \int^{\pi}_{-\pi} [(\Lambda u_{x})u_{x}+(\Lambda u) u_{xx})]_{xx}u_{xxx}dx  \nonumber\\
&\leq (\|\Lambda u_{x}\|_{L^{\infty}[-\pi,\pi]}+\|u_{xx}\|_{L^{\infty}[-\pi,\pi]})\int^{\pi}_{-\pi}u^{2}_{xxx}+(\Lambda u_{xx})u_{xxx}dx.
\end{align}
Thanks to the domain is limited to $[-\pi,\pi]$, the Sobolev embedding theorem and Plancherel formula can be used to show that
 \begin{equation}\label{e22}
\frac{d}{dt}\|u\|_{H^{3}[\pi,\pi]}^{2}\leq c\|u\|_{H^{3}[\pi,\pi]}^{3},
 \end{equation}
 where $c$ is a positive constant.
 Furthermore, a priori bound
 \begin{equation}\label{e23}
 \|u\|_{H^{3}[\pi,\pi]}\leq\frac{\|\varphi\|_{H^{3}[\pi,\pi]}}{1-c\|\varphi\|_{H^{3}[\pi,\pi]}t}
 \end{equation}
 follows.
\end{proof}

We are now in a position to state the blow-up results. It's known that the smooth solution $ u(t,x)$ of (\ref{e13}) can be expanded by
\begin{equation}\label{e24}
 u(t,x)=\sum_{n=-\infty}^{\infty}u_{n}(t)e^{inx},~~u_{n}(t)\in \mathbb{C}.
\end{equation}
Here we assume that the initial data $\varphi\in H^{3}[-\pi,\pi]$ is odd with the form of
\begin{equation}\label{e25}
\varphi(x)=\sum_{n=1}^{\infty}A_{n}sin(nx),
\end{equation}
where $A_{n}$ are constants. In the case $\kappa\geq 0$ and $0<\alpha\leq 1$, we establish that\\

\begin{theorem}\label{th22}(singularity for $\kappa=0$)
Assume there exists a positive number $\delta$ such that $A_{n}\geq\frac{2\delta}{n^{5}}$, then the solution $u(t,x)$ to equation (\ref{e13}) with $\kappa=0$ would blow up at some finite time $T_{1}$, i.e.
$$
\|u(t,\cdot)\|_{L^{2}(-\pi,\pi]}\rightarrow\infty, ~~~as ~~~t\rightarrow T_{1}.
$$
Moreover, there is a rough estimate
$$T_{1}=(\frac{1}{2\delta})^{+}.$$
\end{theorem}

\begin{theorem}\label{th23} (singularity for $\kappa>0, ~\alpha=1$)
Assume there exists a positive number $\delta$ such that $A_{n}\geq\frac{2\delta}{n^{5}}$ and $0<\kappa<\frac{2\delta}{2e-1}$, then the solution $u(t,x)$ to equation (\ref{e13}) would blow up at $T_{2}=\frac{1}{\kappa}$, i.e.
$$
\|u(t,\cdot)\|_{L^{2}(-\pi,\pi]}\rightarrow\infty, ~~~as ~~~t\rightarrow T_{2}=\frac{1}{\kappa}.
$$
\end{theorem}

\begin{theorem}\label{th24} (singularity for $\kappa>0, ~0<\alpha<1$)
Assume there exists a positive number $\delta$ such that $A_{n}\geq\frac{2\delta}{n^{5}}$ and $0<\kappa<\frac{2\delta}{2e-1}$, then the solution $u(t,x)$ to equation (\ref{e13}) would blow up at $T_{2}=\frac{1}{\kappa}$, i.e.
$$
\|u(t,\cdot)\|_{L^{2}(-\pi,\pi]}\rightarrow\infty, ~~~as ~~~t\rightarrow T_{2}=\frac{1}{\kappa}.
$$
\end{theorem}

\begin{remark}
In Theorem \ref{th22}--\ref{th24}, the assumption on initial data $A_{n}\geq\frac{2\delta}{n^{5}}$ is required to be compatible with the convergence of series (\ref{e25}) in $H^{3}[-\pi,\pi]$. (For instance, we can take $A_{n}=\frac{2\delta+1}{n^{5}}$.) In addition, the exponential 5 at denominator in $\frac{2\delta}{n^{5}}$ can be replaced by any larger real number. Essentially, we need to restrict the Fourier coefficients $A_{n}$ of initial data to be positive due to the technique taken here .
\end{remark}

\section{The reduction of the equation (\ref{e13})}

Thanks to the reasonable assumptions of the initial data, the equation (\ref{e13}) can be reduced to a family of ordinary differential equations (ODEs) in this section. Moveover, we can write easily down its implicit solutions.

Firstly, let's reduce the equation in (\ref{e13}). The Hilbert transform of periodic smooth solution $u(t,x)$ in (\ref{e24}) can be expressed by
 \begin{equation}\label{e31}
Hu=-\sum_{n=-\infty}^{\infty}isgn(n)u_{n}(t)e^{inx}=-i\sum_{n=1}^{\infty}u_{n}(t)e^{inx}+i\sum_{n=1}^{\infty}u_{-n}(t)e^{-inx},~u_{n}(t)\in \mathbb{C}.
 \end{equation}
 Then (\ref{e24}) and (\ref{e31}) give
\begin{align}\label{e32}
(Hu)_{x}u_{x}
&=(\sum_{n=1}^{\infty}nu_{n}(t)e^{inx}+\sum_{n=1}^{\infty}nu_{-n}(t)e^{-inx})(i\sum_{n=1}^{\infty}nu_{n}(t)e^{inx}-i\sum_{n=1}^{\infty}nu_{-n}(t)e^{-inx})\nonumber\\
&=i(\sum_{n=1}^{\infty}nu_{n}(t)e^{inx})^{2}-i(\sum_{n=1}^{\infty}nu_{-n}(t)e^{-inx})^{2}.
 \end{align}
Let
\begin{equation}\label{e33}
a(t,x):=(\sum_{n=1}^{\infty}nu_{n}(t)e^{inx})^{2}=\sum_{n=1}^{\infty}a_{n}(t)e^{inx},
\end{equation}
where the coefficients $a_{n}(t)$ can be defined by
\begin{align}\label{e34}
a_{n}(t)
&=\frac{1}{2\pi}\int^{\pi}_{-\pi}a(t,x)e^{-inx}dx  \nonumber\\
&=\frac{1}{2\pi}\int^{\pi}_{-\pi}(\sum_{k=1}^{\infty}\sum_{l=1}^{\infty}lu_{l}e^{ilx}ku_{k}e^{ikx})e^{-inx}dx  \nonumber\\
&=\frac{1}{2\pi}\int^{\pi}_{-\pi}\sum_{k=1}^{\infty}\sum_{l=1}^{\infty}lu_{l}ku_{k}e^{i(l+k-n)x}dx  \nonumber\\
&=\sum_{k=1}^{\infty}\sum_{l=1}^{\infty}lu_{l}ku_{k}\delta(k+l-n)  \nonumber\\
&=\sum_{l=1}^{n-1}lu_{l}(n-l)u_{n-l}.
 \end{align}
It follows from (\ref{e33}) and (\ref{e34}) that
 \begin{equation}\label{e35}
a(t,x)=\sum_{n=1}^{\infty}[\sum_{l=1}^{n-1}lu_{l}(n-l)u_{n-l}]e^{inx}.
\end{equation}
A similar argument shows that
 \begin{equation}\label{e36}
b(t,x):=(\sum_{n=1}^{\infty}nu_{-n}(t)e^{-inx})^{2}=\sum_{n=1}^{\infty}[\sum_{l=1}^{n-1}lu_{-l}(n-l)u_{-n+l}]e^{-inx}.
\end{equation}
Therefore, we can obtain
\begin{align}\label{e37}
(Hu)_{x}u_{x}
&=i(\sum_{n=1}^{\infty}nu_{n}(t)e^{inx})^{2}-i(\sum_{n=1}^{\infty}nu_{-n}(t)e^{-inx})^{2} \nonumber\\
&=\sum_{n=1}^{\infty}[i\sum_{l=1}^{n-1}lu_{l}(n-l)u_{n-l}]e^{inx}-\sum_{n=1}^{\infty}[i\sum_{l=1}^{n-1}lu_{-l}(n-l)u_{-n+l}]e^{-inx}.
 \end{align}
Besides, the diffusion term $\Lambda^{\alpha}u$ also can be expressed by
 \begin{equation}\label{e38}
\Lambda^{\alpha}u(t,x)=\sum_{n=1}^{\infty}n^{\alpha}u_{n}(t)e^{inx}+\sum_{n=1}^{\infty}n^{\alpha}u_{-n}(t)e^{-inx}.
\end{equation}
Taking (\ref{e24}), (\ref{e37}) and (\ref{e38}) into (\ref{e13}) and equating coefficients of $e^{inx}$ and $e^{-inx}$ respectly, we can obtain the following ODEs for $\{u_{n}(t)\}_{n\geq1}$, $\{u_{-n}(t)\}_{n\geq1}$ and $u_{0}(t)$
\begin{equation}\label{e39}
 \left\{ \begin{array}{ll}
 \frac{du_{n}(t)}{dt}-i\sum_{l=1}^{n-1}lu_{l}(t)(n-l)u_{n-l}(t)+\kappa n^{\alpha}u_{n}(t)=0,  \\
 \frac{du_{0}(t)}{dt}=0,\\
 \frac{du_{-n}(t)}{dt}+i\sum_{l=1}^{n-1}lu_{-l}(t)(n-l)u_{-n+l}(t)+\kappa n^{\alpha}u_{-n}(t)=0.
 \end{array} \right.
 \end{equation}

Secondly let's reduce the initial data in (\ref{e13}). It's easy to see that
\begin{align}\label{e310}
\varphi(x)
&=u(0,x)=\sum_{n=-\infty}^{\infty}u_{n}(0)e^{inx}=\sum_{n=-\infty}^{\infty}[iu_{n}(0)sin(nx)+u_{n}(0)cos(nx)]  \nonumber\\
&=\sum_{n=1}^{\infty}[i(u_{n}(0)-u_{-n}(0))sin(nx)]+u_{0}(0)+\sum_{n=1}^{\infty}[(u_{n}(0)+u_{-n}(0))cos(nx)].
 \end{align}
 Comparing the coefficients of the initial data $\varphi(x)$ in (\ref{e25}) with ones in (\ref{e310}), we find
 \begin{equation}\label{e311}
 \left\{ \begin{array}{ll}
 A_{n}=i(u_{n}(0)-u_{-n}(0)),  \\
 u_{0}(0)=0,\\
u_{-n}(0)=-u_{n}(0),
 \end{array} \right.
 \end{equation}
which implies
 \begin{equation}\label{e312}
u_{n}(0)=\left\{ \begin{array}{ll}
 0,~~~~n=0,  \\
\frac{A_{n}}{2i},~~~n\geq 1.
 \end{array} \right.
 \end{equation}

Considering the ODEs (\ref{e39}) with initial condition (\ref{e312}), we can get that
  $u_{0}(t)=0$. It's worth noting that $\{-u_{n}(t)\}$ will solve the third equation in (\ref{e39}) if $\{u_{n}(t)\}$ solve the first equation in (\ref{e39}).
Hence the classical local well-posedness theorem of ODEs would ensure that
 \begin{equation}\label{e313}
u_{-n}(t)=-u_{n}(t).
 \end{equation}
Then the fact $u_{0}(t)=0$ and (\ref{e313}) allow us to write the solution $u(t,x)$ by
\begin{equation}\label{e314}
u(t,x)=\sum_{n=1}^{\infty}2iu_{n}(t)sin(nx).
 \end{equation}
On the other hand, $\{u_{n}(t)\}$ are coefficients of Fourier series (\ref{e24}), which yields
 \begin{equation}\label{e315}
u_{-n}(t)=\overline{u_{n}(t)}.
 \end{equation}
(\ref{e313}) and (\ref{e315}) imply that $\{u_{n}(t)\}$ are pure imaginary numbers with respect to $x$. For convenience, in the following we will take the real-valued functions $\{iu_{n}(t)\}$ as a whole, where $i$ is the imaginary unit. To sum up, if let
 $$
 w_{n}(t)=iu_{n}(t),
 $$
 the ODEs (\ref{e39}) with initial condition (\ref{e312}) would be reduced to the following equations only for $\{w_{n}(t)\}_{n\geq1}$
 \begin{equation}\label{e316}
 \left\{ \begin{array}{ll}
 \frac{dw_{n}(t)}{dt}-\sum_{l=1}^{n-1}lw_{l}(t)(n-l)w_{n-l}(t)+\kappa n^{\alpha}w_{n}(t)=0,  \\
 w_{n}(0)=\frac{A_{n}}{2}.
 \end{array} \right.
 \end{equation}
 The solutions to (\ref{e316}) can be written as
\begin{equation}\label{e317}
w_{n}(t)=\frac{A_{n}}{2}e^{-n^{\alpha}\kappa t}+e^{-n^{\alpha}\kappa t}\int^{t}_{0}e^{n^{\alpha}\kappa s}\sum_{l=1}^{n-1}lw_{l}(s)(n-l)w_{n-l}(s)ds,~~~n\geq1.
\end{equation}

\section{The proof of the main results}

In this section, we mainly give the proofs of blow-up results. Relying on the formal symmetry of solutions in (\ref{e317}), we will finish the proofs by the mathematical induction method.
\subsection{The proof of Theorem \ref{th22}}
\begin{proof}
As $\kappa=0$, (\ref{e317}) gives
\begin{equation}\label{e41}
nw_{n}(t)=\frac{nA_{n}}{2}+n\int^{t}_{0}\sum_{l=1}^{n-1}lw_{l}(s)(n-l)w_{n-l}(s)ds\geq0,~~for~~n\geq1.
\end{equation}
In addition, we claim that $\{nw_{n}(t)\}_{n\geq1}$ satisfy
\begin{equation}\label{e42}
nw_{n}(t)\geq\frac{\delta}{n^{4}}(\frac{1}{2}+\delta t)^{n-1}.
\end{equation}
Now we use the mathematical induction to prove this claim.

If $n=1$, then
$$
w_{1}(t)=\frac{A_{1}}{2}\geq\delta.
$$
As $1\leq l\leq n-1$ , we assume that
\begin{equation}\label{e43}
lw_{l}(t)\geq\frac{\delta}{l^{4}}(\frac{1}{2}+\delta t)^{l-1}
\end{equation}
holds. (\ref{e41}) and (\ref{e43}) imply that
\begin{equation}
\begin{aligned}
nw_{n}(t)
&=\frac{nA_{n}}{2}+n\int^{t}_{0}\sum_{l=1}^{n-1}lw_{l}(s)(n-l)w_{n-l}(s)ds  \nonumber\\
&\geq \frac{nA_{n}}{2}+\delta^{2}n\sum_{l=1}^{n-1}\frac{1}{l^{4}(n-l)^{4}}\int^{t}_{0}(\frac{1}{2}+\delta s)^{n-2}ds \nonumber\\
&\geq \frac{\delta}{n^{4}}+\delta^{2}\frac{n-1}{n^{4}}\int^{t}_{0}(\frac{1}{2}+\delta s)^{n-2}ds \nonumber\\
&=\frac{\delta}{n^{4}}(\frac{1}{2}+\delta t)^{n-1}+\frac{\delta}{n^{4}}(1-\frac{1}{2^{n-1}})\nonumber\\
&\geq \frac{\delta}{n^{4}}(\frac{1}{2}+\delta t)^{n-1},
\end{aligned}
 \end{equation}
and (\ref{e42}) follows.

 (\ref{e314}), (\ref{e42}) and Paserval equality show that
\begin{align}\label{e44}
\|u(t,\cdot)\|^{2}_{L^{2}(-\pi,\pi]}
&=2\pi\sum_{n=1}^{\infty}|2w_{n}(t)|^{2}=8\pi\sum_{n=1}^{\infty}|\frac{1}{n}nw_{n}(t)|^{2}  \nonumber\\
&\geq 8\pi\delta^{2}\sum_{n=1}^{\infty}\frac{[(\frac{1}{2}+\delta t)^{2}]^{n-1}}{n^{10}}.
 \end{align}
Let $g(t):=\frac{1}{2}+\delta t$ be bigger than 1, then the series in (\ref{e44}) would diverge by the basic fact of calculus. Thus, we only need to choose $T_{1}$ is bigger than $\frac{1}{2\delta}$ and the proof is completed.
\end{proof}

\subsection{The proof of Theorem \ref{th23}}
\begin{proof}
In this case, we also have from (\ref{e317})
\begin{equation}\label{e45}
nw_{n}(t)=\frac{nA_{n}}{2}e^{-n\kappa t}+ne^{-n\kappa t}\int^{t}_{0}e^{n\kappa s}\sum_{l=1}^{n-1}lw_{l}(s)(n-l)w_{n-l}(s)ds,~~~n\geq1.
\end{equation}
Similarly we can infer that the following estimate holds
\begin{equation}\label{e46}
nw_{n}(t)\geq\frac{\delta}{n^{4}}(\frac{1}{2}+\delta t)^{n-1}e^{-n\kappa t}.
\end{equation}
Here the mathematical induction method is used again.

As $n=1$, we have
$$
w_{1}(t)=\frac{A_{1}}{2}e^{-\kappa t}\geq\delta e^{-\kappa t}.
$$
For $1\leq l\leq n-1$, we assume that
\begin{equation}\label{e47}
lw_{l}(t)\geq\frac{\delta}{l^{4}}(\frac{1}{2}+\delta t)^{l-1}e^{-l\kappa t}.
\end{equation}
Hence we can obtain that from (\ref{e45}) and (\ref{e47})
\begin{equation}
\begin{aligned}
nw_{n}(t)
&=\frac{nA_{n}}{2}e^{-n\kappa t}+ne^{-n\kappa t}\int^{t}_{0}e^{n\kappa s}\sum_{l=1}^{n-1}lw_{l}(s)(n-l)w_{n-l}(s)ds  \nonumber\\
&\geq \frac{nA_{n}}{2}e^{-n\kappa t}+\delta^{2}n\sum_{l=1}^{n-1}\frac{1}{l^{4}(n-l)^{4}}e^{-n\kappa t}\int^{t}_{0}(\frac{1}{2}+\delta s)^{n-2}ds \nonumber\\
&\geq \frac{\delta}{n^{4}}e^{-n\kappa t}+\delta^{2}\frac{n-1}{n^{4}}e^{-n\kappa t}\int^{t}_{0}(\frac{1}{2}+\delta s)^{n-2}ds \nonumber\\
&=\frac{\delta}{n^{4}}e^{-n\kappa t}(\frac{1}{2}+\delta t)^{n-1}+\frac{\delta}{n^{4}}e^{-n\kappa t}(1-\frac{1}{2^{n-1}})\nonumber\\
&\geq \frac{\delta}{n^{4}}(\frac{1}{2}+\delta t)^{n-1}e^{-n\kappa t},
\end{aligned}
 \end{equation}
and (\ref{e46}) follows. On the other hand, (\ref{e314}), (\ref{e46}) and Paserval equality imply that
\begin{align}\label{e48}
\|u(t,\cdot)\|^{2}_{L^{2}(-\pi,\pi]}
&=2\pi\sum_{n=1}^{\infty}|2w_{n}(t)|^{2}=8\pi\sum_{n=1}^{\infty}|\frac{1}{n}nw_{n}(t)|^{2}  \nonumber\\
&\geq 8\pi\delta^{2}e^{-2\kappa t}\sum_{n=1}^{\infty}\frac{([(\frac{1}{2}+\delta t)e^{-\kappa t}]^{2})^{n-1}}{n^{10}}
 \end{align}
Similarly let $g(t)=(\frac{1}{2}+\delta t)e^{-\kappa t}>1$, then the series in (\ref{e48}) fails to converge. Since $0<\kappa<\frac{2\delta}{2e-1}$, we only need to choose $T_{2}=\frac{1}{\kappa}$ to deduce the occurrence of singularity.
\end{proof}

\subsection{The proof of Theorem \ref{th24}}
\begin{proof}
Similar to the proof of Theorem \ref{th23}, it's sufficient for us to estimate
\begin{equation}\label{e49}
nw_{n}(t)\geq\frac{\delta}{n^{4}}(\frac{1}{2}+\delta t)^{n-1}e^{-n\kappa t}.
\end{equation}
In the case of $n=1$, it's easy to see
$$
w_{1}(t)=\frac{A_{1}}{2}e^{-\kappa t}\geq\delta e^{-\kappa t}.
$$
Then we assume that (\ref{e49}) holds for $l=1,2,...,n-1$,
\begin{equation}\label{e410}
lw_{l}(t)\geq\frac{\delta}{l^{4}}(\frac{1}{2}+\delta t)^{l-1}e^{-l\kappa t}.
\end{equation}
 (\ref{e317}) and (\ref{e410}) yield that
\begin{align}\label{e411}
nw_{n}(t)
&=\frac{nA_{n}}{2}e^{-n^{\alpha}\kappa t}+ne^{-n^{\alpha}\kappa t}\int^{t}_{0}e^{n^{\alpha}\kappa s}\sum_{l=1}^{n-1}lw_{l}(s)(n-l)w_{n-l}(s)ds  \nonumber\\
&\geq \frac{nA_{n}}{2}e^{-n^{\alpha}\kappa t}+\delta^{2}n\sum_{l=1}^{n-1}\frac{1}{l^{4}(n-l)^{4}}e^{-n^{\alpha}\kappa t}\int^{t}_{0}e^{(n^{\alpha}-n)\kappa s}(\frac{1}{2}+\delta s)^{n-2}ds \nonumber\\
&\geq \frac{\delta}{n^{4}}e^{-n\kappa t}+\delta^{2}\frac{n-1}{n^{4}}e^{-n^{\alpha}\kappa t}\int^{t}_{0}e^{(n^{\alpha}-n)\kappa s}(\frac{1}{2}+\delta s)^{n-2}ds \nonumber\\
&=\frac{\delta}{n^{4}}e^{-n\kappa t}+e^{-n\kappa t}\delta^{2}\frac{n-1}{n^{4}}e^{(n-n^{\alpha})\kappa t}\int^{t}_{0}e^{(n^{\alpha}-n)\kappa s}(\frac{1}{2}+\delta s)^{n-2}ds \nonumber\\
&\geq \frac{\delta}{n^{4}}e^{-n\kappa t}+e^{-n\kappa t}\delta^{2}\frac{n-1}{n^{4}}\int^{t}_{0}(\frac{1}{2}+\delta s)^{n-2}ds \nonumber\\
&\geq \frac{\delta}{n^{4}}(\frac{1}{2}+\delta t)^{n-1}e^{-n\kappa t},
 \end{align}
where the second inequality in (\ref{e411}) uses $0<\alpha<1$ and the third inequality in (\ref{e411}) follows from the fact
$$
e^{(n-n^{\alpha})\kappa t}\int^{t}_{0}e^{(n^{\alpha}-n)\kappa s}(\frac{1}{2}+\delta s)^{n-2}ds\geq \int^{t}_{0}(\frac{1}{2}+\delta s)^{n-2}ds,~~for~t\geq0,~0<\alpha<1.
$$
Define
$$h(t)=e^{(n-n^{\alpha})\kappa t}\int^{t}_{0}e^{(n^{\alpha}-n)\kappa s}(\frac{1}{2}+\delta s)^{n-2}ds-\int^{t}_{0}(\frac{1}{2}+\delta s)^{n-2}ds.$$
It's obvious that $h(0)=0$, and a simple computation shows that
\begin{equation}\nonumber
h'(t)= (n-n^{\alpha})\kappa e^{(n-n^{\alpha})\kappa t}\int^{t}_{0}e^{(n^{\alpha}-n)\kappa s}(\frac{1}{2}+\delta s)^{n-2}ds>0,~~for~t\geq0,~0<\alpha<1,
 \end{equation}
then $h(t)\geq0,~~for~t\geq0,~0<\alpha<1$.
\end{proof}

\begin{remark}
The technology in our proof mainly relies on the formal symmetry of ODEs in (\ref{e316}). In fact, if this symmetry is lost, we find that the approach also can be used to deduce the singularity. For example, let's consider following ODEs
\begin{equation}\label{e412}
\frac{dw_{n}(t)}{dt}-\sum_{l=1}^{n-1}lw_{l}(t)w_{n-l}(t)+\kappa n^{\alpha}w_{n}(t)=0.
\end{equation}
Based on the fact
\begin{align}\label{e413}
\sum_{l=1}^{n-1}lw_{l}(t)w_{n-l}(t)
 &=\sum_{l=1}^{n-1}\sqrt{l}w_{l}(t)\sqrt{l}w_{n-l}(t)\nonumber\\
 &=w_{1}w_{n-1}+2w_{2}w_{n-2}+\cdot\cdot\cdot+(n-2)w_{n-2}w_{2}+(n-1)w_{n-1}w_{1}\nonumber\\
 &= w_{1}w_{n-1}+(n-1)w_{n-1}w_{1}+2w_{2}w_{n-2}+(n-2)w_{n-2}w_{2}+\cdot\cdot\cdot \nonumber\\
 &\geq 2(\sqrt{1}\sqrt{n-1})w_{1}w_{n-1}+2(\sqrt{2}\sqrt{n-2})w_{2}w_{n-2}+\cdot\cdot\cdot \nonumber\\
 &=\sum_{l=1}^{n-1}\sqrt{l}w_{l}(t)\sqrt{n-l}w_{n-l}(t)
\end{align}
and the ODEs' comparison principle, we only need to deal with
\begin{equation}\label{e414}
\frac{dw_{n}(t)}{dt}-\sum_{l=1}^{n-1}\sqrt{l}w_{l}(t)\sqrt{n-l}w_{n-l}(t)+\kappa n^{\alpha}w_{n}(t)=0,
\end{equation}
which possess the construction of symmetry. Thus the similar argument in this section indicates that the solutions to (\ref{e414}) would develop into singularities in finite time, hence holds for (\ref{e412}).
\end{remark}
\begin{remark}
Based on these singularity results, it's not difficult to find that smoothing effects due to weak dissipation can't prevent the formation of singularity, thus the global existence is expected when with strong dissipation ($\alpha>1$) for this model. However, here we can't give a super bound for solutions due to the limitation of the spectral method.
\end{remark}

\section*{Acknowledgements}
This project is supported by National Natural Science Foundation of China
(No:11571057).



\begin{thebibliography}{99}


           \bibitem{1} N. Alibaud, J. Droniou and J. Vovelle, {\em Occurrence and non-appearance of shocks in fractal Burgers equations}, J. Hyperbolic Differ. Equ., 4:479-499, 2007.
            \bibitem{2} H. Dong, D. Du and D. Li, {\em Finite time singularities and global well-posedness for fractal Burgers equations}, Indiana Univ. Math. J., 58:807-821, 2009.
            \bibitem{3} A. Kiselev, F. Nazarov and R. Shterenberg, {\em Blow up and regularity for fractal Burgers equation}, Dyn. Partial Differ. Equ., 5:211-240, 2008.
             \bibitem{4} A. C\'{o}rdoba, D. C\'{o}rdoba and M. A. Fontelos, {\em Formation of singularities for a transport equation with nonlocal velocity}, Ann. of Math.(2), 162:1377-1389, 2005.
             \bibitem{5} D. Li and J. Rodrigo, {\em Remarks on a nonlocal transport}, Adv. Math., 374:27pp, 2020.
             \bibitem{6} L. C. F. Ferreira and V. V. C. Moitinho, {\em Global smoothness for a 1D supercritical transport model with nonlocal velocity}, Proc. Amer. Math. Soc., 148:2981-2995, 2020.
             \bibitem{7}  G. R. Baker, X. Li and A. C. Morlet, {\em Analytic structure of two 1D-transport equations with nonlocal fluxes},
             Phys. D, 91:349-375, 1996.
             \bibitem{8} A. C. Morlet, {\em Further properties of a continuum of model equations with globally defined flux}, J. Math.
              Anal. Appl., 221:132-160, 1998.
              \bibitem{9} D. Chae, A. C\'{o}rdoba, D. C\'{o}rdoba and M.A. Fontelos, {\em Finite time singularities in a 1D model of the
              quasi-geostrophic equation}, Adv. Math., 194:203-223, 2005.
             \bibitem{10} A. Castro and D. C\'{o}rdoba, Global existence, {\em singularities and ill-posedness for a nonlocal
             flux}, Adv. Math., 219:1916-1936, 2008.
             \bibitem{11} D. Li and J. Rodrigo, {\em On a one-dimensional nonlocal flux with fractional dissipation}, SIAM. J. Math. Anal., 43:507-526, 2011.
             \bibitem{12} J. Eggers and M. A. Fontelos, {\em Selection of singular solutions in non-local transport equations}, Nonlinearity, 33:325-340, 2020.
             \bibitem{13} S. Schochet, {\em Explicit solutions of the viscous model vorticity equation}, Comm. Pure Appl. Math., 39:531-537, 1986.
             \bibitem{14} A. J. Majda and A. L. Bertozzi, {\em Vorticity and incomprssible flow}, Cambridge University Press, 2002.
             \bibitem{15} T. Kato, {\em On the Cauchy problem for the (generalized) Korteweg-de Vries equation, Studies
                        in Applied Mathematics}, Adv. Math. Suppl. Stud., 8:93-128, 1983.

          \end{thebibliography}
\end{document}